\documentclass[reqno]{amsart}

\usepackage{amssymb}
\usepackage{hyperref}
\usepackage{enumerate}

\theoremstyle{plain}
\newtheorem{Th}{Theorem}
\newtheorem{Le}{Lemma}

\DeclareMathOperator{\supp}{supp}
\DeclareMathOperator{\im}{img}

\newcommand{\vc}{\mathrm{vc}}

\newcommand{\co}{\delta}
\newcommand{\bd}{\partial}

\begin{document}

\title{Harmonic cohomology of symplectic fiber bundles}

\author{Oliver Ebner}

\address{Oliver Ebner,
         Institute of Geometry, TU Graz,
         Kopernikusgasse 24/IV, A-8010 Graz, Austria.}

\email{o.ebner@tugraz.at}

\author{Stefan Haller}

\address{Stefan Haller,
         Department of Mathematics, University of Vienna,
         Nordbergstra{\ss}e 15, A-1090, Vienna, Austria.}

\email{stefan.haller@univie.ac.at}

\thanks{The second author acknowledges the support of the Austrian Science Fund, grant P19392-N13.}
       


\begin{abstract}
We show that every de~Rham cohomology class on the total space of a symplectic fiber bundle with 
closed Lefschetz fibers, admits a Poisson harmonic representative in the sense of Brylinski. 
The proof is based on a new characterization of closed Lefschetz manifolds.
\end{abstract}

\maketitle

\section{Introduction and main result}\label{S:intro}

Suppose $P$ is a Poisson manifold \cite{Va} with Poisson tensor $\pi$. Let $d$ denote the de~Rham differential
on $\Omega(P)$ and write $i_\pi$ for the contraction with the Poisson tensor. Recall that Koszul's \cite{Ko} codifferential $\co:=[i_\pi,d]=i_\pi d-di_\pi$ 
satisfies $\delta^2=0$ and $[d,\delta]=d\co+\co d=0$. 
Differential forms $\alpha\in\Omega(P)$ with $d\alpha=0=\delta\alpha$ are called \emph{(Poisson) harmonic.}
Brylinski \cite{Br} asked for conditions on a Poisson manifold which imply that every de~Rham cohomology class
admits a harmonic representative.

In the symplectic case, this question has been settled by Mathieu. 
Recall that a symplectic manifold $(M,\omega)$ of dimension $2n$ is called \emph{Lefschetz} iff, for all $k$,
\begin{equation*}
[\omega]^k\wedge H^{n-k}(M;\mathbb R)=H^{n+k}(M;\mathbb R).
\end{equation*}
According to Mathieu \cite{Ma}, see \cite{Ya} for an alternative proof, a symplectic manifold is Lefschetz 
iff it satisfies the Brylinski conjecture, i.e.\ every de~Rham cohomology class of $M$ admits a harmonic representative.

In this paper we study the Brylinski problem for smooth symplectic fiber bundles \cite{Mc}.
Recall that the total space of a symplectic fiber bundle $P\to B$ is canonically equipped 
with the structure of a Poisson manifold obtained from the symplectic form on each fiber.
Locally, the Poisson structure on $P$ is product like, that is, every point in $B$ admits 
an open neighborhood $U$ such that there exists a fiber preserving Poisson diffeomorphism 
$P|_U\cong M\times U$. Here $M$ denotes the typical symplectic fiber, equipped with the 
corresponding Poisson structure, and $U$ is considered as a trivial Poisson manifold.
This renders the symplectic foliation of $P$ particularly nice, for its leaves coincide with 
the connected components of the fibers of the bundle $P\to B$.

The aim of this note is to establish the following result, providing a class of Poisson manifolds 
which satisfy the Brylinski conjecture.

\begin{Th}\label{Th:main}
Let $M$ be a closed symplectic Lefschetz manifold, and suppose $P\to B$ 
is a smooth symplectic fiber bundle with typical symplectic fiber $M$. 
Then every de~Rham cohomology class of $P$ admits a Poisson harmonic representative.
Moreover, the analogous statement for compactly supported cohomology holds true.
\end{Th}

This result, as well as a characterization of closed Lefschetz manifolds similar to Theorem~\ref{Th:harext} below, 
has been established in the first author's diploma thesis, employing sightly different methods than those of the 
present work, see \cite{Eb}.

The proof presented in Section~\ref{Se:proof} below is based on a handle body decomposition 
$\emptyset=B_0\subseteq B_1\subseteq B_2\subseteq\cdots$ of $B$.
Given a cohomology class of $P$, we will inductively produce representatives which are harmonic on $P|_{B_k}$, for increasing $k$.
The crucial problem, of course, is to extend harmonic forms across the handle, from $P|_{B_k}$ to $P|_{B_{k+1}}$.
This issue is addressed in Theorem~\ref{Th:harext}, see also Lemma~\ref{Le:elbord}.

\section{Extension of harmonic forms}\label{Sec:harext}

Let $M$ be a closed symplectic manifold and consider the trivial symplectic fiber bundle $P:=M\times\mathbb R^p\times D^q$
where $D^q$ denotes the $q$-dimensional closed unit ball. In other words, the Poisson structure on $P$ is the product 
structure obtained from the symplectic form on $M$ and the trivial Poisson structure on $\mathbb R^p\times D^q$. Note that the
boundary $\partial P=M\times\mathbb R^p\times\partial D^q$ is a Poisson submanifold. 
It turns out that the Lefschetz property of $M$ is equivalent to harmonic extendability of forms, from $\bd P$ to $P$.

To formulate this precisely, we need to introduce some notation which will be used throughout the rest of the paper.
For every Poisson manifold $P$ we let $Z(P):=\{\alpha\in\Omega(P)\mid d\alpha=0\}$ and 
$Z_0(P):=\{\alpha\in\Omega(P)\mid d\alpha=0=\delta\alpha\}$ denote the spaces of closed and harmonic differential forms, respectively.
Moreover, we write $H_0(P):=\ker(d)\cap\ker(\delta)/\im(d)\cap\ker(\delta)$ for the space of de~Rham cohomology 
classes which admit a harmonic representative, $H_0(P)\subseteq H(P)$. If $\iota:S\hookrightarrow P$ is a Poisson
submanifold, then the relative complex $\Omega(P,S):=\{\alpha\in\Omega(P)\mid\iota^*\alpha=0\}$ is invariant under $\delta$,
and we define the relative harmonic cohomology $H_0(P,S)\subseteq H(P,S)$ in an analogous manner.
Finally, if $Q$ is a Poisson manifold and $B$ is a smooth manifold we let $\Omega_\vc(Q\times B)$ denote the 
space of forms with vertically compact support (with respect to the projection $Q\times B\to Q$), and 
define the harmonic cohomology with vertically compact supports
$H_{\vc,0}(Q\times B)\subseteq H_\vc(Q\times B)$ in the obvious way.

Here is the main result that will be established in this section.

\begin{Th}\label{Th:harext}
Let $M$ be a closed symplectic manifold, suppose $p,q\in\mathbb N_0$, and consider the Poisson manifold
$P:=M\times\mathbb R^p\times D^q$. Then the following are equivalent:
\begin{enumerate}[(i)]
\item\label{Th:harext:i} 
$M$ is Lefschetz, i.e.\ $H_0(M)=H(M)$ according to \cite{Ma}.
\item\label{Th:harext:iii} 
$H_0(P,\partial P)=H(P,\partial P)$.
\item\label{Th:harext:v} 
$H_{\vc,0}(P\setminus\partial P)=H_\vc(P\setminus\partial P)$ with respect to the projection along $D^q\setminus\partial D^q$.
\item\label{Th:harext:ii} 
If $\alpha\in Z(P)$ is harmonic on a neighborhood of $\partial P$, then there exists $\beta\in\Omega(P)$, supported on $P\setminus\partial P$, so that $\alpha+d\beta$ is harmonic on $P$.
\item\label{Th:harext:iv} 
If $\alpha\in Z(P)$ and $\delta\iota^*\alpha=0$, then there exists $\beta\in\Omega(P)$ with $\iota^*\beta=0$, so that $\alpha+d\beta$ is harmonic on $P$.
Here $\iota:\partial P\hookrightarrow P$ denotes the canonical inclusion.
\end{enumerate}
\end{Th}

An essential ingredient for the proof of Theorem~\ref{Th:harext} is the following $d\delta$-Lemma.

\begin{Le}[$d\delta$-Lemma, \cite{Gu,Me}]\label{Le:ddelta}
A closed symplectic manifold is Lefschetz if and only if $\ker(\co)\cap\im(d)=\im(d\co)$.
\end{Le}

We will also make use of the following averaging argument.

\begin{Le}\label{Le:inv}
Suppose $G$ is a connected compact Lie group acting on a Poisson manifold $P$ via Poisson diffeomorphisms, and let
$r:\Omega(P\times I)\to\Omega(P\times I)^G$, $r(\alpha):=\int_Gg^*\alpha\,dg$, 
denote the standard projection onto the space of $G$-invariant forms, $I:=[0,1]$.
Then there exists an operator $A:\Omega(P\times I)\to\Omega(P\times I)$, commuting with $d$, $i_\pi$ and $\delta$,
so that $A(\alpha)=\alpha$ in a neighborhood of $P\times\{1\}$ and $A(\alpha)=r(\alpha)$ in a neighborhood 
of $P\times\{0\}$, for all $\alpha\in\Omega(P\times I)$.
\end{Le}

\begin{proof} 
Choose finitely many smoothly embedded closed balls $D_i\subseteq G$ such that $\bigcup_i\mathring D_i=G$.
Let $\lambda_i$ denote a partition of unity on $G$ so that $\supp(\lambda_i)\subseteq D_i$.
Choose smooth contractions $h_i:D_i\times I\to G$ so that $h_i(g,t)=g$ for $t\leq1/3$
and $h_i(g,t)=e$ for $t\geq2/3$, $g\in D_i$. Here $e$ denotes the neutral element of $G$.
Using the maps 
\begin{equation*}
\phi_{i,g}:P\times I\to P\times I,\quad \phi_{i,g}(x,t):=(h_i(g,t)\cdot x,t),\qquad g\in D_i,
\end{equation*}
we define the operator $A:\Omega(P\times I)\to\Omega(P\times I)$ by
\begin{equation*}
A(\alpha):=\sum_i\int_{D_i}\lambda_i(g)\phi_{i,g}^*\alpha\,\,dg
\end{equation*}
where integration is with respect to the invariant Haar measure of $G$.
It is straightforward to verify that $A$ has the desired properties, the relations $[A,i_\pi]=0=[A,\delta]$
follow from the fact that each $\phi_{i,g}$ is a Poisson map. 
\end{proof}

The following application of Lemma~\ref{Le:inv} will be used in the proof of Theorem~\ref{Th:main}.

\begin{Le}\label{Le:MA}
Let $M$ be a symplectic manifold and consider the Poisson manifold $P:=M\times\mathbb R^p\times A^q$ where 
$A^q:=\{\xi\in\mathbb R^q\mid\frac12\leq\xi\leq1\}$ denotes the $q$-dimensional annulus.
Moreover, suppose $\alpha\in\Omega(P)$ is harmonic on a neighborhood of $\partial_+P:=M\times\mathbb R^p\times\partial D^q$. 
Then there exist $\beta\in\Omega(P)$, supported on $P\setminus\partial_+P$, and $\beta_1,\beta_2\in Z_0(M)$, so that
$\tilde\alpha:=\alpha+d\beta$ is harmonic on $P$, and $\tilde\alpha=\sigma^*\beta_1+\sigma^*\beta_2\wedge\rho^*\theta$
in a neighborhood of $\partial_-P:=M\times\mathbb R^p\times\frac12\partial D^q$. Here 
$\sigma:P\to M$ and $\rho:P\to\partial D^q$ denote the canonical projections, and $\theta$ denotes the standard 
volume form on $\partial D^q$.\footnote{To be specific, in the case $q=1$ we assume $\theta(-1)=-1/2$ and $\theta(1)=1/2$, so that 
$\int_{\partial D^q}\theta=1$ with respect to orientation on $\partial D^q$ induced from the standard orientation of $D^q$.}
\end{Le}

\begin{proof}
W.l.o.g.\ we may assume $\alpha\in Z_0(P)$ and $\alpha=\tau^*\gamma$ in a neighborhood of $\partial_-P$ where 
$\gamma\in Z_0(M\times\partial D^q)$ and $\tau=(\sigma,\rho):P\to M\times\partial D^q$ denotes the canonical projection.
Applying the operator $A$ from Lemma~\ref{Le:inv} to $\alpha$, we obtain $\tilde\alpha\in Z_0(P)$
so that $\tilde\alpha=\alpha$ in a neighborhood of $\partial_+P$, and $\tilde\alpha=\tau^*\tilde\gamma$ in a neighborhood
of $\partial_-P$, where $\tilde\gamma\in Z_0(M\times\partial D^q)$ is $SO(q)$-invariant. We conclude that $\tilde\gamma$
is of the form $\tilde\gamma=\beta_1+\beta_2\wedge\theta$ with $\beta_1,\beta_2\in Z_0(M)$, whence 
$\tilde\alpha=\sigma^*\beta_1+\sigma^*\beta_2\wedge\rho^*\theta$ in a neighborhood of $\partial_-P$.
Clearly, there exists $\beta\in\Omega(P)$, supported on $P\setminus\partial_+P$, such that $\tilde\alpha-\alpha=d\beta$.
\end{proof}

\begin{Le}\label{Le:fibint}
Let $P$ be a Poisson manifold, and suppose $B$ is an oriented smooth manifold with boundary.
Then integration along the fibers $\int_B:\Omega_\vc(P\times B)\to\Omega(P)$ commutes with $i_\pi$ and $\delta$. 
\end{Le}

\begin{proof}
The relation $i_\pi\int_B\alpha=\int_Bi_\pi\alpha$ is obvious. Combining this with Stokes' theorem, that is
$[\int_B,d]=\int_{\partial B}\iota^*$, we obtain
\begin{equation*}
\textstyle
[\int_B,\co]=[\int_B,[i_\pi,d]]=[[\int_B,i_\pi],d]+[i_\pi,[\int_B,d]]=[i_\pi,\int_{\partial B}\iota^*]=0.
\end{equation*}
Here $\iota:P\times \partial B\hookrightarrow P\times B$ denotes the canonical inclusion.
\end{proof}

\begin{Le}\label{Le:Liso}
Suppose $Q$ is a Poisson manifold, and consider the Poisson manifold $P:=Q\times D^q$.
Then the Thom (K\"unneth) isomorphism restricts to an isomorphism of harmonic cohomology, i.e.\
$H_0^{*-q}(Q)=H^*_{\vc,0}(P\setminus\partial P)=H_0^*(P,\partial P)$.
\end{Le}

\begin{proof}
Choose $\eta\in\Omega^q(D^q)$, supported on $D^q\setminus\partial D^q$, such that $\int_{D^q}\eta=1$. 
Clearly, the chain map $\Omega(Q)\to\Omega_\vc(P\setminus\partial P)\subseteq\Omega(P,\partial P)$, 
$\alpha\mapsto\alpha\wedge\eta$, commute with $\delta$.
This map induces the Thom isomorphism which therefore preserve harmonicity. Its inverse is induced by
integration along the fibers $\int_{D^q}:\Omega(P,\partial P)\to\Omega(Q)$, and this commutes with $\delta$ too, see Lemma~\ref{Le:fibint}.
\end{proof}

Now the table is served and we proceed to the

\begin{proof}[Proof of Theorem~\ref{Th:harext}]
Set $Q:=M\times\mathbb R^p$ and note that the isomorphism $H(Q)=H(M)$ induced by the canonical projection
restricts to an isomorphism of harmonic cohomology $H_0(Q)=H_0(M)$.
The equivalence of the first three statements thus follows from Lemma~\ref{Le:Liso}.
Let us continue by showing that (\ref{Th:harext:v}) implies (\ref{Th:harext:ii}).
Assume $\alpha\in Z(P)$ is harmonic on a neighborhood of $\partial P$.
Let $\rho:P\setminus(M\times\mathbb R^p\times\{0\})\to\partial D^q$ and $\sigma:P\to M$ denote the canonical projections.
In view of Lemma~\ref{Le:MA}, we may w.l.o.g.\ assume $\alpha=\sigma^*\beta_1+\sigma^*\beta_2\wedge\rho^*\theta$
in a neighborhood of $\partial P$
where $\beta_1,\beta_2\in Z_0(M)$ and $\theta$ denotes the standard volume form on $\partial D^q$.
Using Stokes' theorem for integration along the fiber of $M\times D^q\to M$, we obtain
\begin{equation*}
\beta_2=\int_{\partial D^q}j^*\alpha=-d\int_{D^q}j^*\alpha\in\im(d)\cap\ker(\co)
\end{equation*}
where $j:M\times D^q\to M\times\{0\}\times D^q\subseteq P$ denotes the canonical inclusion.
By the $d\delta$-Lemma~\ref{Le:ddelta}, we thus have $\beta_2=d\co\gamma$ for some differential form $\gamma$ on $M$. 
Let $\lambda$ be a smooth function on $P$, identically $1$ in a neighborhood of $\partial P$, identically $0$
near $M\times\mathbb R^p\times\{0\}$, and constant in the $M$-direction.
Then $\tilde\alpha:=\sigma^*\beta_1+d(\co\sigma^*\gamma\wedge\lambda\rho^*\theta)$ is a harmonic on $P$, and $\alpha-\tilde\alpha=0$ in a neighborhood of $\partial P$.
Hence, using (\ref{Th:harext:v}), we find $\beta\in\Omega(P)$, supported on $P\setminus\partial P$, so that $\alpha-\tilde\alpha+d\beta$ 
is harmonic on $P$. Thus, $\beta$ has the desired property.
Let us next show that (\ref{Th:harext:ii}) implies (\ref{Th:harext:iv}). Suppose $\alpha\in Z(P)$ and $\delta\iota^*\alpha=0$.
Clearly, there exists $\beta_1\in\Omega(P)$, with $\iota^*\beta_1=0$, so that $\tilde\alpha:=\alpha+d\beta_1$ satisfies 
$r^*\tilde\alpha=\tilde\alpha$ near $\partial P$, where $r:P\setminus(M\times\mathbb R^p\times\{0\})\to\partial P$
denotes the canonical radial retraction. Particularly, $\tilde\alpha$ is harmonic on a neighborhood of $\partial P$. According to (\ref{Th:harext:ii})
there exists $\beta_2\in\Omega(P)$, supported on $P\setminus\partial P$, so that $\tilde\alpha+d\beta_2$ is harmonic on $P$. The form $\beta:=\beta_1+\beta_2$
thus has the desired property.
Obviously, (\ref{Th:harext:iv}) implies (\ref{Th:harext:iii}).
\end{proof}

\section{Proof of Theorem~\ref{Th:main}}\label{Se:proof}

Choose a proper Morse function $f$ on $B$, bounded from below, so 
that the preimage of each critical value consists of a single critical point \cite{MiM}.
We label the critical values in increasing order 
$c_0<c_1<\cdots$, and choose regular values $r_k$ so that $c_{k-1}<r_k<c_k$.
By construction, the sublevel sets $B_k:=\{f(x)\leq r_k\}$ provide an increasing filtration of $B$
by compact submanifolds with boundary, $\emptyset=B_0\subseteq B_1\subseteq B_2\subseteq\cdots$. 
The statement in Theorem~\ref{Th:main} is an immediate consequence of the following

\begin{Le}\label{Le:elbord}
Suppose $\alpha\in Z(P)$ is a closed form which is harmonic on a neighborhood of $P|_{B_k}$.
Then there exists $\beta\in\Omega(P)$, supported on $P|_{B_{k+2}\setminus B_k}$, such that
$\alpha+d\beta$ is harmonic on a neighborhood of $P|_{B_{k+1}}$.
\end{Le}

\begin{proof}
Let $q$ denote the Morse index of the unique critical point in $B_{k+1}\setminus B_k$, and set $p:=\dim B-q$. 
Recall \cite{MiM} that there exists an embedding $j:\mathbb R^p\times D^q\to B_{k+1}\setminus\mathring B_k$ so that 
$j(\mathbb R^p\times\partial D^q)=j(\mathbb R^p\times D^q)\cap\partial B_k$.
Moreover, there exists a vector field $X$ on $B$, supported on $B_{k+2}\setminus B_k$, so that its flow
$\varphi_t$ maps $B_{k+1}$ into any given neighborhood of $\partial B_k\cup j(\{0\}\times D^q)$, 
for sufficiently large $t$.

Trivializing the symplectic bundle $P$ over the image of $j$, we obtain an isomorphism of Poisson
manifolds $j^*P\cong M\times\mathbb R^p\times D^q$. Using Theorem~\ref{Th:harext}(\ref{Th:harext:ii}), we may thus assume that
there exists an open neighborhood $U$ of $\partial B_k\cup j(\{0\}\times D^q)$ so that $\alpha$ is harmonic on $P|_U$.
Let $\tilde X$ denote the horizontal lift of $X$ with respect to a symplectic connection \cite{Mc} on $P$,
and denote its flow at time $t$ by $\tilde\varphi_t$. Clearly, each $\tilde\varphi_t$ is a Poisson map.
Moreover, there exists $t_0$ so that $\tilde\varphi_{t_0}$ maps $P|_{B_{k+1}}$ into $P|_U$.
Thus, $\tilde\varphi_{t_0}^*\alpha$ is harmonic on $P|_{B_{k+1}}$.
Furthermore, $\tilde\varphi_{t_0}^*\alpha-\alpha=d\beta$ where $\beta:=\int_0^{t_0}\tilde\varphi_t^*i_{\tilde X}\alpha\,dt$
is supported on $P|_{B_{k+2}\setminus B_k}$.
\end{proof}

\end{document}